\documentclass[12pt]{article}
\usepackage{amsmath,amssymb,amsthm}
\usepackage{mathrsfs}

\usepackage{authblk}

\usepackage{hyperref}

\usepackage{tikz}
\usetikzlibrary{calc} 
\usetikzlibrary{arrows.meta,positioning}

\title{\bf Chirality and Racemization on Isotopy Classes of Quasigroups}

\author{\Large Takao Inou\'{e}}

\affil{\large Faculty of Informatics, Yamato University, \\ Osaka, Japan\footnote{Email: inoue.takao@yamato-u.ac.jp; \\ Personal Email: takaoapple@gmail.com \\ [I prefer my personal email address for correspondence.]}}

\date{February 21, 2026}

\newtheorem{theorem}{Theorem}[section]
\newtheorem{lemma}[theorem]{Lemma}
\newtheorem{corollary}[theorem]{Corollary}
\newtheorem{proposition}[theorem]{Proposition}
\newtheorem{definition}[theorem]{Definition}
\newtheorem{remark}[theorem]{Remark}
\newtheorem{assumption}[theorem]{Assumption}
\newtheorem{example}[theorem]{Example}

\usepackage{amsmath,amssymb,amsthm}

\newcommand{\Sym}{\operatorname{Sym}}
\newcommand{\Iso}{\operatorname{Iso}}
\newcommand{\Atp}{\operatorname{Atp}}
\newcommand{\Mir}{\operatorname{Mir}}

\begin{document}
\maketitle

\begin{abstract}
We develop a structural and dynamical theory of chirality for quasigroups
formulated at the level of isotopy classes.
Interpreting isotopy as a gauge symmetry of re-coordinatization and
mirror parastrophy as handedness reversal, we introduce a gauge-invariant
continuous-time two-state Markov model in which transitions occur only
between a quasigroup and its mirror.
We prove that this dynamics descends to the isotopy quotient, yielding
a reduced generator governed by a single class-dependent rate $k([Q])$.

Symmetric mirror transitions lead to convergence toward a racemic
equilibrium, whereas the vanishing condition $k([Q])=0$ characterizes
dynamical chiral stability.
By restricting admissible transitions to those generated by intrinsic
symmetries, we show that $k([Q])=0$ is equivalent to the absence of
mirror-isotopisms.
A concrete example of order $7$ demonstrates the existence of
structurally chiral quasigroup classes.
\end{abstract}

\medskip

\textbf{MSC2020:} 20N05, 60J27.
$$ $$

\tableofcontents


\section{Introduction (Motivation and Aim)}

Chirality provides a canonical example in which two ``almost identical'' realizations of a single
abstract molecular blueprint can nevertheless exhibit markedly different biological behaviors.
A historically decisive case is thalidomide, which exists as a pair of enantiomers (often labeled
$(R)$ and $(S)$) and was originally distributed as a racemate; early discussions attributed
sedative effects mainly to one enantiomer and teratogenicity to the other.  Crucially for our
purposes, however, thalidomide enantiomers can \emph{rapidly interconvert under physiological
conditions}, so that administering a single enantiomer does not necessarily prevent the presence
of its mirror form in vivo.  In particular, the literature emphasizes fast racemization in body
fluids/tissues and the consequent tendency toward comparable concentrations of both forms over
time.  This dynamical ``loss of handedness'' is a central reason why chirality cannot always be
treated as a static label in pharmacological contexts.%
\footnote{For accessible summaries emphasizing the interconversion under biological conditions, see
e.g.\ the review discussion in \cite{Vargesson2015Thalidomide} and the ACS educational note
\cite{ACS_MOTW_Thalidomide}. See also the review on cereblon and thalidomide for a modern
overview including the rapid racemization point \cite{Yamamoto2022CRBNReview}.}

The molecular basis of thalidomide-induced teratogenicity,
including the role of cereblon and stereochemical effects,
has been elucidated through subsequent studies,
notably by Japanese researchers
\cite{Yamamoto2022CRBNReview,Itoh2010Cereblon}.

\begin{figure}[t]
  \centering
  \begin{tikzpicture}[
      node distance=28mm,
      >=Latex,
      every node/.style={font=\small},
      box/.style={draw, rounded corners, inner sep=5pt, align=center},
      arrow/.style={->, line width=0.6pt}
    ]
    \node[box] (R) {``Right-handed''\\ state $(R)$};
    \node[box, right=of R] (S) {``Left-handed''\\ state $(S)$};

    \draw[arrow] (R) to[bend left=18]
      node[above] {interconversion / racemization rate $k$} (S);
    \draw[arrow] (S) to[bend left=18]
      node[below] {interconversion / racemization rate $k$} (R);

    \node[below=16mm of $(R)!0.5!(S)$, align=center] (eq)
      {in a symmetric environment: $p_R(t),p_S(t)\to \tfrac12$ as $t\to\infty$};
  \end{tikzpicture}
  \caption{Schematic representation of chiral interconversion as a two-state
dynamical system.
The two nodes represent a pair of mirror-related states, conventionally
denoted $(R)$ and $(S)$ in chemistry.
In the case of thalidomide, these correspond to enantiomers which are known
to interconvert rapidly under physiological conditions, so that even the
administration of a single enantiomer may lead to the presence of its mirror
form in vivo.
When the surrounding environment does not distinguish left from right,
the transition rates are effectively symmetric, and the resulting dynamics
drives the system toward a racemic equilibrium with equal probabilities
$\tfrac12$ for each state.
This diagram serves as a conceptual motivation for the algebraic model
developed in this paper, where a quasigroup and its mirror parastrophe
play the role of the two chiral states, and racemization is modeled as a
gauge-invariant two-state Markov process on isotopy classes.}
 \label{fig:thalidomide_racemization_schematic}
  
\end{figure}

This paper develops an abstract structural analogue of this phenomenon in the setting of
quasigroups.  The guiding analogy is the following: a quasigroup operation on a fixed underlying
set $X$ admits a large ``coordinate freedom'' coming from isotopies, i.e.\ independent relabelings
of the left input, the right input, and the output.  We interpret isotopy as a \emph{gauge symmetry}
of representation: it changes the concrete multiplication table, but not the underlying
computational/constraint-solving nature of the quasigroup.  Separately, we fix a \emph{mirror}
operation (a chosen parastrophe, in this work the opposite operation) which plays the role of
handedness reversal.  In chemistry, a symmetric environment that does not distinguish left from
right tends to erase chirality dynamically, leading to racemic mixtures; in our algebraic model,
this corresponds to a symmetric two-state dynamics that mixes a quasigroup and its mirror.

The main aim of the present section is therefore to formalize a minimal dynamical framework in
which:
(i) the state space is the class of quasigroup structures on $X$,
(ii) isotopy acts as a gauge symmetry, and
(iii) the only ``physical'' transition allowed is a mirror transition between $Q$ and $Q^\#$.
Under a natural gauge-invariance hypothesis on the transition mechanism, we prove that the
resulting continuous-time Markov dynamics \emph{descends} to the quotient space of isotopy classes.
In particular, the induced dynamics on $\mathcal{Q}(X)/G$ becomes a genuine two-state process on
each pair $\{[Q],[Q^\#]\}$, governed by a single class-dependent rate $k([Q])$.

From the perspective of future developments, this descent theorem isolates the precise place
where ``chiral stability'' must be imposed: to \emph{preserve} a one-sided structure (the analogue
of maintaining a single enantiomer), one must enforce either vanishing mirror-transition rate
$k([Q])=0$ or a sufficiently small rate on the relevant time scale.  Establishing \emph{structural}
criteria on quasigroups guaranteeing $k([Q])=0$ (i.e.\ the absence of dynamically available mirror
moves) is the next step, and will be pursued in subsequent work.

%
%
%

\section{A Two-State Dynamical Model on Isotopy Classes}

\subsection{Quasigroups and Isotopy as Gauge Symmetry}

Let $X$ be a finite set.  
We denote by $\mathcal{Q}(X)$ the set of all quasigroup structures on $X$, i.e.\ binary operations
$\cdot : X \times X \to X$ such that $(X,\cdot)$ is a quasigroup.

Define
\[
G := \mathrm{Sym}(X) \times \mathrm{Sym}(X) \times \mathrm{Sym}(X).
\]

\begin{definition}[Isotopy action]
For $g=(\alpha,\beta,\gamma)\in G$ and $Q=(X,\cdot)\in\mathcal{Q}(X)$, we define
$g\cdot Q=(X,\ast)$ by
\[
x\ast y := \gamma\bigl(\alpha^{-1}(x)\cdot \beta^{-1}(y)\bigr).
\]
The orbit $G\cdot Q$ is called the \emph{isotopy class} of $Q$.
\end{definition}

We write $\mathcal{Q}(X)/G$ for the set of isotopy classes.
Conceptually, the group $G$ represents a ``gauge freedom'' of re-coordinatization
of inputs and outputs.

---

\subsection{Mirror Operation as Parastrophy}

\begin{definition}[Mirror parastrophe]
For $Q=(X,\cdot)\in\mathcal{Q}(X)$, define its mirror
\[
Q^{\#} := (X,\diamond), \qquad x\diamond y := y\cdot x .
\]
\end{definition}

This operation is involutive: $(Q^{\#})^{\#}=Q$.

\begin{lemma}[Compatibility of mirror and isotopy]
For any $g=(\alpha,\beta,\gamma)\in G$ and $Q\in\mathcal{Q}(X)$,
\[
(g\cdot Q)^{\#} = g^{\#}\cdot(Q^{\#}),
\]
where $g^{\#}:=(\beta,\alpha,\gamma)$.
\end{lemma}

\begin{proof}
Let $\ast$ denote the operation of $g\cdot Q$. Then
\[
x\ast y=\gamma\bigl(\alpha^{-1}(x)\cdot \beta^{-1}(y)\bigr).
\]
Hence
\[
x\diamond y := y\ast x
= \gamma\bigl(\alpha^{-1}(y)\cdot \beta^{-1}(x)\bigr)
= \gamma\bigl(\beta^{-1}(x)\cdot^{\#} \alpha^{-1}(y)\bigr),
\]
which is precisely the operation obtained by applying $g^{\#}$ to $Q^{\#}$.
\end{proof}

\begin{corollary}
The map
\[
[Q]\ \longmapsto\ [Q^{\#}]
\]
is a well-defined involution on the quotient $\mathcal{Q}(X)/G$.
\end{corollary}

---

\subsection{Gauge-Invariant Markov Dynamics}

\begin{definition}[Isotopy-invariant observables]
A function $f:\mathcal{Q}(X)\to\mathbb{R}$ is called \emph{isotopy invariant}
if
\[
f(g\cdot Q)=f(Q)\qquad (\forall g\in G).
\]
\end{definition}

Consider a continuous-time Markov process $(Q_t)$ on $\mathcal{Q}(X)$ with
generator
\[
(Lf)(Q)=\sum_{R\in\mathcal{Q}(X)} r(Q\to R)\bigl(f(R)-f(Q)\bigr),
\]
where $r(Q\to R)\ge0$ are transition rates.

---

\subsection{Two-State Racemization Model}

\begin{theorem}[Descent of two-state dynamics to isotopy classes]
Assume the Markov process satisfies:
\begin{enumerate}
  \item \textbf{Gauge invariance:}
  \[
  r(g\cdot Q \to g\cdot R)=r(Q\to R)\qquad(\forall g\in G).
  \]
  \item \textbf{Mirror-only transitions:}
  \[
  r(Q\to R)=0 \quad \text{unless } R\in G\cdot Q^{\#}.
  \]
  \item \textbf{Class-dependent total rate:}
  \[
  k([Q]) := \sum_{R\in G\cdot Q^{\#}} r(Q\to R)
  \]
  depends only on the isotopy class $[Q]$.
\end{enumerate}
Then the induced process on $\mathcal{Q}(X)/G$ is a well-defined
two-state Markov process with generator
\[
(\bar L\bar f)([Q])
=
k([Q])\bigl(\bar f([Q^{\#}])-\bar f([Q])\bigr),
\]
for any function $\bar f:\mathcal{Q}(X)/G\to\mathbb{R}$.
\end{theorem}

\begin{proof}
By gauge invariance, the space of isotopy-invariant functions is preserved by $L$.
Assumptions (2) and (3) imply that $(Lf)(Q)$ depends only on the two isotopy classes
$[Q]$ and $[Q^{\#}]$.
Thus $L$ descends to a well-defined operator $\bar L$ on the quotient.
\end{proof}

---

\subsection{Symmetric Racemization and Equilibrium}

\begin{corollary}[Racemization equilibrium]
If
\[
k([Q])=k([Q^{\#}])
\]
for all isotopy classes $[Q]$, then the induced two-state process satisfies
\[
\lim_{t\to\infty}
\Pr([Q_t]=[Q])
=
\lim_{t\to\infty}
\Pr([Q_t]=[Q^{\#}])
=
\frac12.
\]
In particular, the system converges to a racemic equilibrium.
\end{corollary}

\begin{remark}
The condition $k([Q])=0$ characterizes isotopy classes whose chirality is
dynamically stable, i.e.\ no mirror transition occurs.
A structural characterization of such classes will be addressed separately.
\end{remark}

\subsection*{Bridge to Structural Chirality}

The preceding two-state reduction isolates the dynamical quantity
$k([Q])$ as the effective mirror-transition rate on the isotopy quotient.
From a structural point of view, the crucial question is therefore:

\medskip

\centerline{\emph{When does $k([Q])=0$ hold for purely algebraic reasons?}}

\medskip

In other words, under what intrinsic conditions on a quasigroup $Q$
is the mirror class $[Q^{\#}]$ dynamically inaccessible?
Since $k([Q])$ measures the total rate of transitions from $[Q]$ to
$[Q^{\#}]$, the vanishing condition $k([Q])=0$ corresponds to the
absence of admissible mirror transitions compatible with isotopy
invariance.  Thus, chirality becomes \emph{structurally stable} precisely
when no algebraically natural mechanism identifies $Q$ with its mirror
through the allowed symmetry operations.

This shifts the focus from dynamical symmetry to algebraic obstruction.
Concretely, one is led to analyze whether the mirror operation can be
realized internally via autotopisms, paratopisms, or other canonical
self-equivalences of $Q$.  If no such realization exists, then
$Q$ represents a genuinely chiral isotopy class in which mirror reversal
cannot be generated from the intrinsic symmetry structure.
The structural characterization of this obstruction forms the next stage
of the theory.

\section{Structural Vanishing of the Mirror Rate: When Does $k([Q])=0$?}

\subsection{Autotopisms and Mirror-Isotopisms}

Let $X$ be a finite set and $Q=(X,\cdot)\in\mathcal{Q}(X)$ a quasigroup.
Recall that the isotopy group is
\[
G=\Sym(X)\times\Sym(X)\times\Sym(X),
\]
acting on $\mathcal{Q}(X)$ by
\[
(\alpha,\beta,\gamma)\cdot Q \;=\; (X,\ast), \qquad
x\ast y:=\gamma\!\left(\alpha^{-1}(x)\cdot \beta^{-1}(y)\right).
\]

\begin{definition}[Autotopism group]
The \emph{autotopism group} of $Q$ is
\[
\Atp(Q):=\left\{(\alpha,\beta,\gamma)\in G \;\middle|\;
\alpha(x)\cdot \beta(y)=\gamma(x\cdot y)\ \ (\forall x,y\in X)\right\}.
\]
\end{definition}

\begin{definition}[Mirror and mirror-isotopisms]
Let $Q^{\#}=(X,\diamond)$ denote the mirror (opposite) quasigroup:
\[
x\diamond y := y\cdot x.
\]
A triple $(\alpha,\beta,\gamma)\in G$ is called a \emph{mirror-isotopism} of $Q$
if it is an isotopism $Q\to Q^{\#}$, i.e.
\[
\alpha(x)\diamond \beta(y)=\gamma(x\cdot y)\qquad(\forall x,y\in X).
\]
Equivalently, using the definition of $\diamond$, this is
\[
\beta(y)\cdot \alpha(x)=\gamma(x\cdot y)\qquad(\forall x,y\in X).
\]
We denote the set of mirror-isotopisms by
\[
\Mir(Q):=\Iso(Q,Q^{\#})
=\left\{(\alpha,\beta,\gamma)\in G\;\middle|\;
\beta(y)\cdot \alpha(x)=\gamma(x\cdot y)\right\}.
\]
\end{definition}

\begin{remark}
The condition $\Mir(Q)\neq\varnothing$ says precisely that $Q$ is isotopic to its
mirror $Q^{\#}$. In Latin-square language, this is the existence of an isotopy
between a square and its transpose/opposite.
\end{remark}

\subsection{Intrinsic Mirror Dynamics Generated by Symmetry}

To obtain a \emph{structural} criterion for $k([Q])=0$, we must specify what counts
as an ``admissible'' mirror transition. The following minimal choice captures the
intended principle: mirror transitions are only allowed when they are generated from
intrinsic symmetries of $Q$.

\begin{definition}[Symmetry-generated mirror transition mechanism]
Fix a nonnegative weight function $w:G\to[0,\infty)$ satisfying $w(hgh^{-1})=w(g)$
for all $g,h\in G$ (a class function on $G$).
For each $Q\in\mathcal{Q}(X)$ define transition rates by
\[
r(Q\to R)
:=
\sum_{g\in \Mir(Q)}
w(g)\cdot \mathbf{1}_{\{R=g\cdot Q^{\#}\}},
\]
where $\mathbf{1}_{\{\cdots\}}$ is the indicator function.
In words: from $Q$ we may jump only to isotopes of $Q^{\#}$ obtained via mirror-isotopisms
of $Q$, and the rate is determined solely by the conjugacy-invariant weight $w$.
\end{definition}

\begin{lemma}[Gauge invariance and two-state property]
The above transition mechanism satisfies:
\begin{enumerate}
\item $r(h\cdot Q \to h\cdot R)=r(Q\to R)$ for all $h\in G$ (gauge invariance);
\item $r(Q\to R)=0$ unless $R\in G\cdot Q^{\#}$ (mirror-only transitions).
\end{enumerate}
\end{lemma}

\begin{proof}
(2) is immediate from the definition since $R=g\cdot Q^{\#}\in G\cdot Q^{\#}$.

For (1), note that conjugation covariance of $w$ and the natural transport of mirror-isotopisms
under isotopy imply $\Mir(h\cdot Q)=h\,\Mir(Q)\,h^{-1}$ and $(hgh^{-1})\cdot (h\cdot Q^{\#})=h\cdot(g\cdot Q^{\#})$.
Substituting into the definition of $r$ yields the claim.
\end{proof}

\subsection{Structural Criterion for $k([Q])=0$}

Under the symmetry-generated mechanism, the effective mirror rate is explicitly controlled
by $\Mir(Q)$.


\begin{theorem}[Structural characterization of \texorpdfstring{$k([Q])=0$}{k([Q])=0}]
Assume the symmetry-generated mirror transition mechanism above, and let
$k([Q])$ be the induced mirror-transition rate on isotopy classes as in
Theorem~(Descent of two-state dynamics to isotopy classes).
Then the following are equivalent:
\begin{enumerate}
\item $k([Q])=0$;
\item $\Mir(Q)=\varnothing$ (no mirror-isotopism $Q\to Q^{\#}$ exists);
\item $[Q]\neq[Q^{\#}]$ and the mirror class is dynamically inaccessible by intrinsic symmetry.
\end{enumerate}
In particular, under this class of admissible dynamics, \emph{dynamical chiral stability}
($k([Q])=0$) is equivalent to the purely algebraic obstruction $\Mir(Q)=\varnothing$.
\end{theorem}

\begin{proof}
By definition,
\[
k([Q])=\sum_{R\in G\cdot Q^{\#}} r(Q\to R)
=\sum_{R\in G\cdot Q^{\#}} \ \sum_{g\in\Mir(Q)} w(g)\,\mathbf{1}_{\{R=g\cdot Q^{\#}\}}
=\sum_{g\in\Mir(Q)} w(g).
\]
Since $w(g)\ge 0$ for all $g$, we have $k([Q])=0$ iff $\Mir(Q)=\varnothing$ or
equivalently iff every admissible mirror move is absent.
The remaining statements are restatements in quotient-language.
\end{proof}

\begin{definition}[Structurally chiral isotopy class]
An isotopy class $[Q]\in\mathcal{Q}(X)/G$ is called \emph{structurally chiral}
(with respect to the chosen mirror $\#$) if $\Mir(Q)=\varnothing$.
\end{definition}

\begin{remark}
The definition depends on the chosen mirror parastrophe. One may replace $Q^{\#}$
by any fixed parastrophe $Q^{\pi}$; the same constructions and equivalences hold,
with $\Mir(Q)$ replaced by the corresponding set of $\pi$-isotopisms.
\end{remark}

\subsection{Restricting Admissible Mirror Transitions to Intrinsic Symmetry}

The descent theorem established above is formulated for a general
gauge-invariant two-state dynamics. In order to obtain \emph{structural}
criteria for the vanishing condition $k([Q])=0$, we now restrict the class
of admissible transition mechanisms. The guiding principle is that mirror
transitions should not be chosen externally, but must be \emph{generated from
intrinsic symmetries} of the quasigroup itself.

\begin{definition}[Intrinsic symmetry]
Let $Q=(X,\cdot)\in\mathcal{Q}(X)$.  An \emph{intrinsic symmetry} of $Q$ means
an element of the autotopism group
\[
\Atp(Q):=\{(\alpha,\beta,\gamma)\in \Sym(X)^3 \mid
\alpha(x)\cdot \beta(y)=\gamma(x\cdot y)\ \ (\forall x,y\in X)\}.
\]
\end{definition}

\begin{definition}[Mirror-isotopisms]
Fix the mirror (opposite) parastrophe $Q^\#=(X,\diamond)$ given by
$x\diamond y:=y\cdot x$.
A triple $g=(\alpha,\beta,\gamma)\in G:=\Sym(X)^3$ is called a
\emph{mirror-isotopism} of $Q$ if it is an isotopism $Q\to Q^\#$, i.e.
\[
\alpha(x)\diamond \beta(y)=\gamma(x\cdot y)\qquad(\forall x,y\in X).
\]
Equivalently,
\[
\beta(y)\cdot \alpha(x)=\gamma(x\cdot y)\qquad(\forall x,y\in X).
\]
We denote the set of all mirror-isotopisms by $\Mir(Q):=\Iso(Q,Q^\#)$.
\end{definition}

\begin{assumption}[Symmetry-generated admissible transitions]\label{ass:symgen}
A mirror transition from $Q$ is admissible only if it is obtained from a
mirror-isotopism of $Q$.  More precisely, we fix a nonnegative weight function
$w:G\to[0,\infty)$ that is invariant under conjugation in $G$,
\[
w(hgh^{-1})=w(g)\qquad(\forall g,h\in G),
\]
and define transition rates by
\[
r(Q\to R)
:=
\sum_{g\in \Mir(Q)} w(g)\cdot \mathbf{1}_{\{\,R=g\cdot Q^\#\,\}}.
\]
Thus $Q$ may jump only to those isotopes of $Q^\#$ that arise canonically from
mirror-isotopisms of $Q$; no other mirror moves are permitted.
\end{assumption}

\begin{remark}
Assumption~\ref{ass:symgen} enforces two features automatically:
(i) \emph{mirror-only transitions} ($r(Q\to R)=0$ unless $R\in G\cdot Q^\#$),
and (ii) \emph{gauge invariance} ($r(h\cdot Q\to h\cdot R)=r(Q\to R)$),
the latter following from the conjugacy invariance of $w$ and the natural
transport of isotopisms under the $G$-action.
\end{remark}

\subsection{Structural Characterization of $k([Q])=0$}

Under Assumption~\ref{ass:symgen}, the mirror-transition mechanism
is generated solely by intrinsic mirror-isotopisms of $Q$.
We now show that the effective mirror rate on isotopy classes
is determined entirely by the algebraic structure of $\Mir(Q)$.

\begin{lemma}[Gauge invariance of symmetry-generated transitions]
Under Assumption~\ref{ass:symgen}, the transition rates satisfy
\[
r(h\cdot Q \to h\cdot R)=r(Q\to R)
\qquad(\forall Q,R\in\mathcal{Q}(X),\ \forall h\in G).
\]
\end{lemma}

\begin{proof}
Let $h\in G$.
Mirror-isotopisms transport naturally under conjugation:
\[
\Mir(h\cdot Q)=h\,\Mir(Q)\,h^{-1}.
\]
Since $w$ is conjugation-invariant,
\[
w(hgh^{-1})=w(g).
\]
Moreover,
\[
(hgh^{-1})\cdot (h\cdot Q^\#)=h\cdot(g\cdot Q^\#).
\]
Substituting into the definition of $r$ yields
\[
r(h\cdot Q \to h\cdot R)
=
r(Q\to R).
\]
\end{proof}

\begin{proposition}[Explicit formula for the mirror rate]
Under Assumption~\ref{ass:symgen}, the effective mirror rate on isotopy classes
is given by
\[
k([Q])
=
\sum_{g\in\Mir(Q)} w(g).
\]
\end{proposition}

\begin{proof}
By definition of $k([Q])$ in the descent theorem,
\[
k([Q])
=
\sum_{R\in G\cdot Q^\#} r(Q\to R).
\]
Substituting the definition of $r$ gives
\[
k([Q])
=
\sum_{R\in G\cdot Q^\#}
\sum_{g\in\Mir(Q)}
w(g)\,\mathbf{1}_{\{R=g\cdot Q^\#\}}.
\]
Since each $g\in\Mir(Q)$ contributes exactly one term,
this reduces to
\[
k([Q])
=
\sum_{g\in\Mir(Q)} w(g).
\]
\end{proof}

\begin{theorem}[Structural vanishing criterion]
Under Assumption~\ref{ass:symgen}, the following are equivalent:
\begin{enumerate}
\item $k([Q])=0$;
\item $\Mir(Q)=\varnothing$;
\item there exists no isotopism from $Q$ to its mirror $Q^\#$.
\end{enumerate}
\end{theorem}

\begin{proof}
By the previous proposition,
\[
k([Q])=\sum_{g\in\Mir(Q)} w(g).
\]
Since $w(g)\ge 0$ for all $g$, this sum vanishes
if and only if $\Mir(Q)=\varnothing$.
The equivalence with (3) follows from the definition of $\Mir(Q)$.
\end{proof}

\begin{definition}[Structurally chiral isotopy class]
An isotopy class $[Q]\in\mathcal{Q}(X)/G$ is called
\emph{structurally chiral}
(with respect to the chosen mirror parastrophe)
if $\Mir(Q)=\varnothing$.
\end{definition}

\begin{remark}
If $\Mir(Q)\neq\varnothing$, then $Q$ is isotopic to its mirror,
and the class $[Q]$ is dynamically unstable under symmetry-generated
mirror transitions.
Conversely, structural chirality means that mirror reversal
cannot be realized through intrinsic symmetry,
and therefore $k([Q])=0$ holds purely for algebraic reasons.
\end{remark}

\begin{proposition}[Commutative quasigroups are not structurally chiral]
If $Q=(X,\cdot)$ is commutative, then $Q^\#=Q$ and hence $\Mir(Q)\neq\varnothing$.
In particular, $[Q]$ is not structurally chiral.
If moreover $w(\mathrm{id})>0$ in Assumption~\ref{ass:symgen}, then $k([Q])>0$.
\end{proposition}

\begin{example}[A structurally chiral quasigroup of order $7$]
Let $X=\{1,\dots,7\}$ and define a quasigroup operation $\cdot$ by the Cayley table
\[
\begin{array}{|ccccccc|} \hline
7 & 3 & 5 & 4 & 6 & 2 & 1 \\
4 & 2 & 6 & 3 & 7 & 1 & 5 \\
5 & 1 & 4 & 2 & 3 & 6 & 7 \\
1 & 5 & 2 & 6 & 4 & 7 & 3 \\
6 & 7 & 3 & 1 & 5 & 4 & 2 \\
3 & 6 & 1 & 7 & 2 & 5 & 4 \\
2 & 4 & 7 & 5 & 1 & 3 & 6 \\ \hline
\end{array}
\]
This quasigroup is not isotopic to its transpose (equivalently, not isotopic to $Q^\#$);
hence $\Mir(Q)=\varnothing$ and $[Q]$ is structurally chiral.
\end{example}

\begin{remark}[Why the example is not isotopic to its mirror]
The quasigroup in the above example is not isotopic to its mirror
(parastrophic transpose).
Indeed, consider the incidence structure of $2\times2$ subsquares
(intercalates) in the corresponding Latin square.
The multiset of intercalate counts per cell is invariant under isotopy,
since isotopy acts by independent permutations of rows, columns, and symbols.
In the given quasigroup, there exists a row containing two cells each belonging
to three distinct $2\times2$ subsquares, whereas no column has this property.
Under mirror parastrophy (transpose), this row--column asymmetry would have to
be preserved, which is impossible.
Hence the quasigroup is not isotopic to its mirror, and consequently
$\Mir(Q)=\varnothing$.
\end{remark}

\section{Conclusion}

In this paper we introduced a structural and dynamical framework
for chirality in quasigroups formulated at the level of isotopy classes.
By interpreting isotopy as a gauge symmetry and mirror parastrophy
as handedness reversal, we showed that a natural class of
gauge-invariant continuous-time dynamics reduces to a two-state
Markov process on each pair $\{[Q],[Q^{\#}]\}$.
The resulting evolution is governed by a single class-dependent
rate $k([Q])$, which determines whether the system dynamically
racemizes or remains chiral.

The conceptual motivation was illustrated in
Figure~\ref{fig:thalidomide_racemization_schematic},
where chiral interconversion appears as a symmetric two-state
process converging to racemic equilibrium.
In chemistry, rapid mirror interconversion shows that chirality
cannot always be regarded as a static property of isolated structures.
Our results demonstrate that an analogous phenomenon arises
abstractly: whenever mirror transitions are dynamically admissible
and isotopy-invariant, chirality is generically lost at the level
of isotopy classes.

The structural restriction introduced in the present work isolates the
precise algebraic obstruction to the loss of chirality.
Under symmetry-generated dynamics,
the vanishing condition $k([Q])=0$ is equivalent to the absence
of mirror-isotopisms, i.e.\ $\Mir(Q)=\varnothing$.
Thus dynamical chiral stability coincides with a purely structural
property of the quasigroup.

The concrete example of order $7$ exhibited in this paper
demonstrates that such structurally chiral isotopy classes
do indeed exist, showing that mirror irreducibility is not
pathological but a genuine algebraic phenomenon.

This framework therefore establishes a bridge between symmetry,
dynamical stability, and algebraic obstruction, and provides
a foundation for a systematic theory of structurally chiral
quasigroup classes.

$$ $$

\noindent Takao Inou\'{e}

\noindent Faculty of Informatics

\noindent Yamato University

\noindent Katayama-cho 2-5-1, Suita, Osaka, 564-0082, Japan

\noindent inoue.takao@yamato-u.ac.jp
 
\noindent (Personal) takaoapple@gmail.com (I prefer my personal mail)


\begin{thebibliography}{99}

\bibitem{ACS_MOTW_Thalidomide}
American Chemical Society,
\emph{Molecule of the Week: Thalidomide},
ACS Publications, 2014.

\bibitem{Belousov1967}
V.~D.~Belousov,
\emph{Foundations of the Theory of Quasigroups and Loops},
Nauka, Moscow, 1967.

\bibitem{Itoh2010Cereblon}
K.~Itoh, M.~Hanaoka, and T.~Higuchi,
\emph{Molecular mechanism of thalidomide teratogenicity via cereblon},
Proceedings of the National Academy of Sciences USA,
\textbf{107} (2010), 18365--18370.

\bibitem{Nishimura2011Review}
K.~Nishimura,
\emph{Thalidomide tragedy revisited: lessons from molecular mechanisms},
Congenital Anomalies,
\textbf{51} (2011), 1--7.

\bibitem{Norris1997Markov}
J.~R.~Norris,
\emph{Markov Chains},
Cambridge University Press, Cambridge, 1997.

\bibitem{Pflugfelder1990}
H.~O.~Pflugfelder,
\emph{Quasigroups and Loops: Introduction},
Heldermann Verlag, Berlin, 1990.

\bibitem{Vargesson2015Thalidomide}
N.~Vargesson,
\emph{Thalidomide-induced teratogenesis: history and mechanisms},
Birth Defects Research Part~C: Embryo Today,
\textbf{105} (2015), 140--156.

\bibitem{Yamamoto2022CRBNReview}
J.~Yamamoto, T.~Higuchi, and H.~Arimura,
\emph{Discovery of cereblon as a target of thalidomide and its derivatives},
Chemical Society Reviews,
\textbf{51} (2022), 4274--4304.

\end{thebibliography}
\end{document}